\documentclass[11pt]{amsart}
\usepackage{float}
\usepackage{amsfonts}
\usepackage{amsmath}
\usepackage{amssymb}
\usepackage{graphicx}
\usepackage[left=2.5cm, right=2.5cm, top=3cm]{geometry}
\usepackage{mathtools}
\usepackage{amsthm}
\usepackage{latexsym}
\usepackage{fancyhdr}
\usepackage{array}
\usepackage{amscd}
\usepackage{lscape}
\usepackage{tikz}
\usepackage{tikz-cd}
\usepackage{float}
\usepackage{bm}
\usepackage{subfigure}
\usepackage{grffile}
\usepackage{array}
\usepackage[tableposition=above]{caption}
\usepackage{longtable}
\usepackage{booktabs}
\usepackage{adjustbox}
\usepackage{lipsum}
\usepackage{makecell}
\usepackage{multirow}
\newcolumntype{C}[1]{>{\centering\arraybackslash}p{#1}}
\usepackage[toc,page, title, titletoc]{appendix}
\definecolor{navy}{HTML}{2F729C}
\definecolor{red1}{HTML}{FF0000} % this is red
\usepackage[hyperfootnotes=false, colorlinks, linkcolor={blue}, citecolor={magenta}, filecolor={blue}, urlcolor={blue}, plainpages=false, pdfpagelabels]{hyperref}
\newtheorem{theorem}{Theorem}[section]
\newtheorem{lemma}[theorem]{Lemma}
\newtheorem{proposition}[theorem]{Proposition}
\newtheorem{corollary}[theorem]{Corollary}
\theoremstyle{definition}

\newtheorem{conjecture}[theorem]{Conjecture}

\newtheorem{mainthm}{Theorem}

\numberwithin{equation}{section}

\title{On $abc$ triples of the form $(  1,c-1,c)  $}

\author{Elise Alvarez-Salazar}
\address{Department of Mathematics, University of California, Santa Barbara, CA 93106 USA}
\email{ealvarez-salazar@ucsb.edu}

\author{Alexander J. Barrios}
\address{Department of Mathematics, University of St. Thomas, St. Paul, MN 55105 USA}
\email{abarrios@stthomas.edu}

\author{Calvin Henaku}
\address{Department of Mathematics, University of Michigan, Ann Arbor, MI, USA}
\email{chenaku@umich.edu}

\author{Summer Soller}
\address{Department of Mathematics, Colorado State University, Fort Collins, CO 80523 USA}
\email{summer.soller@colostate.edu}



\begin{document}
\begin{abstract}
By an $abc$ triple, we mean a triple $(a,b,c)$ of relatively prime positive integers $a,b,$ and $c$ such that $a+b=c$ and $\operatorname{rad}(abc)<c$, where $\operatorname{rad}(n)$ denotes the product of the distinct prime factors of $n$. The study of $abc$ triples is motivated by the $abc$ conjecture, which states that for each $\epsilon>0$, there are finitely many $abc$ triples $(a,b,c)$ such that $\operatorname{rad}(abc)^{1+\epsilon}<c$. The necessity of the $\epsilon$ in the $abc$ conjecture is demonstrated by the existence of infinitely many $abc$ triples. For instance, $\left(  1,9^{k}-1,9^{k}\right)  $ is an $abc$ triple for each positive integer $k$. In this article, we study $abc$ triples of the form $\left(1,c-1,c\right)  $ and deduce two general results that allow us to recover existing sequences of $abc$ triples having $a=1$ that are in the literature.
\end{abstract}

\maketitle
%\setcounter{tocdepth}{1}
%\tableofcontents

\section{Introduction}

In $1985$, Masser and Oesterl\'{e} proposed the $abc$ conjecture
\cite{MR992208, MR3731300}, which states:

\begin{conjecture}[The $abc$ conjecture]For every $\epsilon>0$, there are finitely many relatively
prime positive integers $a,b,$ and $c$ with $a+b=c$ such that%
\[
\operatorname{rad}(abc)^{1+\epsilon}<c,
\]
where $\operatorname{rad}(n)$ denotes the product of the distinct prime
factors of a positive integer~$n$.
\end{conjecture}

Due to its profound implications, this simple-to-state conjecture is one of
the most important open questions in number theory. For instance, some
consequences of the $abc$ conjecture include an asymptotic version of Fermat's
Last Theorem, Faltings's Theorem, Roth's Theorem, and Szpiro's Conjecture
\cite{MR1141316, MR1005184, MR992208}. For further information on the $abc$ conjecture, see the excellent survey article \cite{MR3584566}.

The statement of the $abc$ conjecture naturally leads us to ask if the
$\epsilon$ is necessary. This leads us to the \textquotedblleft simplistic
$abc$ conjecture,\textquotedblright\ which asks if there are finitely many
relatively prime positive integers $a,b,$ and $c$ with $a+b=c$ for which
$\operatorname{rad}\!\left(  abc\right)  <c$. We call such triples,  {\em abc triples}. The \textquotedblleft simplistic $abc$
conjecture\textquotedblright\ is false, as demonstrated by the triple $\left(
1,3^{2^{k}}-1,3^{2^{k}}\right)  $, which is an $abc$ triple for each positive
integer $k$. This infinite sequence of $abc$ triples is one of the first
documented counterexamples to the simplistic $abc$ conjecture and was
communicated to Lang \cite{MR1005184} by Jastrzebowski and Spielman. A theorem
of Stewart \cite{MR781193} leads to similar sequences of $abc$
triples such as $\left(  1,8^{7^{k}}-1,8^{7^{k}}\right)  $, where $k$ is a
positive integer \cite{MR3584566}. Jastrzebowski and Spielman's counterexample
can also be recovered from the following result: for each odd prime $p$
and each positive integer $k$, $\left(  1,p^{(p-1)k}-1,p^{(p-1)k}\right)  $ is
an $abc$ triple \cite{MR4474850}. Another construction, due to Granville and
Tucker \cite{MR1930670}, shows that for each odd prime $p$, $\left(
1,2^{p\left(  p-1\right)  }-1,2^{p\left(  p-1\right)  }\right)  $ is an $abc$ triple.

In this article, we prove that $(1,c-1,c)$ is an $abc$ triple if and only if $\operatorname{cosocle}(c-~1)>\operatorname{rad}(c)$, where $\operatorname{cosocle}(m)=\frac{m}{\operatorname{rad}(m)}$ for $m$ a positive integer (see Proposition~\ref{abcprop}). We note that the term cosocle is borrowed from module theory, where the \textit{cosocle} of an $R$-module $M$ is the maximal semisimple quotient of $M$, or equivalently, $\frac{M}{\operatorname{rad}(M)}$. In our setting, the cosocle plays a crucial role in our results, from which we recover each of the above mentioned sequences of $abc$ triples. To provide context for our work, we note that the equivalence above requires us to compute $\operatorname{cosocle}(c-1)$ in order to deduce whether $(1,c-1,c)$ is an $abc$ triple. The computation of $\operatorname{cosocle}(c-1)$ requires knowledge of the prime factorization of $c-1$, which becomes computationally difficult as c gets large. Our main results provide a recipe for constructing infinitely many $abc$ triples of the form $(1,c-1,c)$ based on knowledge of a divisor of $c-1$ or $c$. Our first theorem illustrates this.

\begin{mainthm}
\label{mainthm1} \textit{Let $c$ and $m$ be positive integers with $c>1$. If $m$ divides $c-1$ and $\operatorname{cosocle}(m)>\operatorname{rad}(c)$, then $\left(  1,c^{k}-1,c^{k}\right)  $ is an $abc$ triple for each positive
integer $k$.}.
\end{mainthm}

We prove Theorem~\ref{mainthm1} in Section~\ref{section2}. While the proof is
elementary, the result allows us to recover each of the previously mentioned
sequences of $abc$ triples. It also leads to new sequences of $abc$ triples,
such as $\left(  1,n^{\left(  n-1\right)  k}-1,n^{\left(  n-1\right)
k}\right)  $ which is an $abc$ triple for each positive integer $k$ whenever~$n$ is a positive integer that is either odd or even and non-squarefree (see
Corollary \ref{maincor1}). A slight modification of the proof of Theorem
\ref{mainthm1} leads us to our next result (which is also proven in Section~\ref{section2}).

\begin{mainthm}\label{mainthm2}
\textit{Let $b$ and $m$ be positive integers. If $m$ divides $b+1$ and $\operatorname{cosocle}(m)>\operatorname{rad}(b)$, then $\left(  1,b^{k},b^{k}+1\right)  $ is an $abc$ triple for each positive odd integer $k$.}
\end{mainthm}

A consequence of Theorem~\ref{mainthm1} is that if $(1,c-1,c)$ is an $abc$ triple, then $(1,c^k-1,c^k)$ is an $abc$ triple for each positive integer $k$ (see Corollary~\ref{cortri}). Similarly, we obtain from Theorem~\ref{mainthm2} that if $(1,b,b+1)$ is an $abc$ triple, then $(1,b^k,b^k+1)$ is an $abc$ triple for each odd integer $k$ (see Corollary~\ref{transferkodd}). These results lead to the following question: given an integer $c>1$, for what positive integers $k$ is $(1,c^k-1,c^k)$ an $abc$ triple? We answer this question with Theorem~\ref{gransuggestion}, which provides necessary and sufficient conditions to determine those integers $k$ which yield an $abc$ triple of the form $(1,c^k-1,c^k)$.

In Section~\ref{section3}, we demonstrate various consequences of
Theorems~\ref{mainthm1} and \ref{mainthm2}. For example, we prove that if
$n>1$ is an integer and $p$ is an odd prime such that $p>\operatorname{rad}%
\!\left(  n\right)  $, then $\left(  1,n^{p\left(  p-1\right)  k}-1,n^{p\left(
p-1\right)  k}\right)  $ is an $abc$ triple for each positive integer $k$ (see
Corollary \ref{granville-tucker}). In particular, taking $\left(  n,k\right)
=\left(  2,1\right)  $ allows us to recover Granville and Tucker's original
construction \cite{MR1930670}. Another consequence is the following: if
$n\geq3$ is an odd integer and $b=n^{j}-1$ for some positive integer $j$, then
$\left(  1,b^{nk},b^{nk}+1\right)  $ is an $abc$ triple for each positive odd
integer $k$ (see Corollary \ref{cor3_11}). Taking $\left(  n,j\right)
=\left(  3,1\right)  $ gives us that $\left(  1,8^{k},8^{k}+1\right)  $ is an
$abc$ triple for each odd integer $k$. 

We conclude the article with Section \ref{section4}, which is an analysis of
the $abc$ triples found by the ABC@Home Project of the form $\left(
1,c-1,c\right)  $ with $c<10^{18}$. The ABC@Home Project was a network
computing project that was started in $2006$ by the Mathematics Department of
Leiden University, together with the Dutch Kennislink Science Institute. By
$2011$, they found that there are exactly $14\hspace{0.15em}482\hspace{0.15em}065$ $abc$ triples
$\left(  a,b,c\right)  $ with $c<10^{18}$. By the time the project came to a
close in $2015$, the ABC@Home Project had found a total of $23\hspace{0.15em}827\hspace{0.15em}716$
$abc$ triples $\left(  a,b,c\right)  $ with $c<2^{63}$. We note that this list
is not exhaustive of all $abc$ triples with $c<2^{63}$. In particular, the
ABC@Home project found that there are exactly $45\hspace{0.15em}604$ $abc$ triples of the
form $\left(  1,c-1,c\right)  $ with $c<10^{18}$. Further observations about the $abc$ triples found by the ABC@Home Project can be found in \cite[Chapter~7]{palen}.

Motivated by the results in Section~\ref{section3}, we study those $abc$
triples found by the ABC@Home Project that are of the form $(1,n^{l}-1,n^{l})$
or $(1,n^{l},n^{l}+1)$ for some integer $l>1$. We find that this amounts to
$8\hspace{0.15em}413$ $abc$ triples. For $abc$ triples $\left(  1,c-1,c\right)  $ of the
aforementioned form, we show that approximately $48.7\%$ of the $abc$ triples
with $c\leq 10^{6}$ can be obtained from the results proven in Section
\ref{section3}. We also find that for $abc$ triples of the form $(1,n^{l}%
-1,n^{l})$, there are only four cases where there does not exists a proper
divisor $m$ of $n^{l}-1$ for which $\operatorname{cosocle}\!\left(  m\right)
>\operatorname{rad}\!\left(  n\right)  $.

\section{Main Results\label{section2}}

In this section, we establish Theorems \ref{mainthm1} and \ref{mainthm2}. To
do so, we recall the following elementary property about the radical of a
positive integer.

\begin{lemma}\label{elemlemma}
Let $m$ and $n$ be relatively prime positive integers. Then
$\operatorname{rad}\!\left(  mn\right)  =\operatorname{rad}\!\left(  m\right)
\operatorname{rad}\!\left(  n\right)  $ and $\operatorname{rad}\!\left(
m\right)  \leq m$. Moreover, $\operatorname{rad}\!\left(
m^k\right) =\operatorname{rad}\!\left(
m\right)$ for each positive integer $k$.
\end{lemma}

We will assume Lemma~\ref{elemlemma} implicitly throughout this work. Next, we show an important facet about $abc$ triples of the form $\left(
1,c-1,c\right)  $, which showcases the importance of the cosocle in our arguments.

\begin{proposition}
\label{abcprop}Let $c>1$ be an integer. Then the following are equivalent:
\begin{enumerate}
    \item[($i$)] $\operatorname{cosocle}(c-1)>\operatorname{rad}(c)$;
    \item[($ii$)] $\operatorname{cosocle}(c)>\operatorname{rad}(c-1)$;
    \item[($iii$)] $\left(  1,c-1,c\right)  $ is an $abc$ triple.
\end{enumerate}
\end{proposition}

\begin{proof}
Suppose that $\operatorname{rad}\!\left(  c\right)  <\operatorname{cosocle}%
\!\left(  c-1\right)  $. From the equalities $\operatorname{rad}\!\left(  c\right)
=\frac{c}{\operatorname{cosocle}(  c)  }$ and
$\operatorname{cosocle}\!\left(  c-1\right)  =\frac{c-1}{\operatorname{rad}%
(  c-1)  }$, we deduce that%
\begin{align*}
\operatorname{rad}\!\left(  c\right)  <\operatorname{cosocle}\!\left(
c-1\right)  \qquad & \text{if and only if}\qquad\frac{c}{\operatorname{cosocle}%
\!\left(  c\right)  }<\frac{c-1}{\operatorname{rad}\!\left(  c-1\right)  }\\
& \text{if and only if}\qquad\operatorname{rad}\!\left(  c-1\right)  <\frac
{c-1}{c}\operatorname{cosocle}\!\left(  c\right)  .
\end{align*}
Since $\frac{c-1}{c}<1$, we have the desired inequality: $\operatorname{rad}%
\!\left(  c-1\right)  <\operatorname{cosocle}\!\left(  c\right)  $.

Next, suppose that $\frac{\operatorname{rad}(c-1)}{\operatorname{cosocle}%
(c)}<1$. Since $\operatorname{rad}\!\left(  c\right)  =\frac{c}%
{\operatorname{cosocle}(c)}$, we observe that%
\[
\operatorname{rad}(c(c-1))=\operatorname{rad}(c)\operatorname{rad}%
(c-1)=\frac{\operatorname{rad}(c-1)}{\operatorname{cosocle}(c)}c<c,
\]
which shows that $\left(  1,c-1,c\right)  $ is an $abc$ triple.

Lastly, if $\left(  1,c-1,c\right)  $ is an $abc$ triple, then
$\operatorname{rad}(c(c-1))<c$. Consequently,%
$$c >\operatorname{rad}(c(c-1))=\operatorname{rad}(c)\operatorname{rad}%
(c-1)=\frac{\operatorname{rad}(c)(c-1)}{\operatorname{cosocle}(c-1)}.$$
This implies that
$$ \operatorname{rad}(c)<\operatorname{cosocle}
(c-1)\frac{c}{c-1}.$$
Since $\operatorname{rad}(c)$ is an integer and $\frac{c}{c-1}>1$, we deduce
that $\operatorname{rad}(c)\leq\left\lfloor \operatorname{cosocle}%
(c-1)\frac{c}{c-1}\right\rfloor $, where $\left\lfloor x\right\rfloor $
denotes the floor function. Since $\frac{\operatorname{cosocle}(c-1)}{c-1}<1$,
we observe that
\begin{align*}
\left\lfloor \operatorname{cosocle}(c-1)\frac{c}{c-1}\right\rfloor  &
=\left\lfloor \operatorname{cosocle}(c-1)+\frac{\operatorname{cosocle}%
(c-1)}{c-1}\right\rfloor \\
&  =\operatorname{cosocle}(c-1).
\end{align*}
Lastly, $c$ is relatively prime to $c-1$, and thus $\operatorname{cosocle}%
(c-1)>\operatorname{rad}(c)$.
\end{proof}

An automatic consequence of Proposition~\ref{abcprop} is that if $c$ or $c-1$ is squarefree, then~$(1,c-1,c)$ is not an $abc$ triple since the cosocle of a squarefree positive integer is $1$. 
Our next result establishes that the radical of a positive integer $n$ is preserved if $n$ is divided by the cosocle of any of its divisors.

\begin{lemma}
\label{Lemmaonrad}Let $m$ and $n$ be positive integers. If $m$ divides $n$,
then $\operatorname{rad}(n)=\operatorname{rad}\!\left(  \frac{n}%
{\operatorname{cosocle}(m)}\right)  $.
\end{lemma}

\begin{proof}
If $m=1$, there is nothing to show. So suppose that $m>1$ and let
$m=\prod_{i=1}^{r}p_{i}^{e_{i}}$ be the unique prime factorization of $m$,
with each $p_{i}$ denoting a distinct prime. Since $m$ divides $n$, we have
that $n=q\prod_{i=1}^{r}p_{i}^{f_{i}}$ where $e_{i}\leq f_{i}$ for $1\leq i\leq
r$ and $q$ is relatively prime to $m$. Since $\operatorname{cosocle}(m)=\prod_{i=1}%
^{r}p_{i}^{e_{i}-1}$, we deduce that
\[
\frac{n}{\operatorname{cosocle}(m)}=q\prod_{i=1}^{r}p_{i}^{f_{i}-e_{i}+1}.
%.
\]
For $1\leq i\leq r$, observe that $f_{i}-e_{i}+1\geq1$ and thus
$\operatorname{rad}\!\left(  \frac{n}{\operatorname{cosocle}(m)}\right)
=\operatorname{rad}(n)$.
\end{proof}

With this lemma, we are now ready to prove Theorem \ref{mainthm1}.

\begin{proof}
[Proof of Theorem \ref{mainthm1}.]Since $c^{k}-1=\left(  c-1\right)  \sum
_{j=0}^{k-1}c^{j}$, we deduce that $m$ divides
$c^{k}-1$ for each positive integer $k$. By Lemma \ref{Lemmaonrad},
$\operatorname{rad}\!\left(  c^{k}-1\right)  =\operatorname{rad}\!\left(
\frac{c^{k}-1}{\operatorname{cosocle}(m)}\right)  $. By assumption,
$\frac{\operatorname{rad}(c)}{\operatorname{cosocle}(m)}<~1$ and thus 
\[
\operatorname{rad}\!\left(  c^{k}\left(  c^{k}-1\right)  \right)
=\operatorname{rad}(c)\operatorname{rad}\!\left(  \frac{c^{k}-1}%
{\operatorname{cosocle}(m)}\right)  \leq\frac{\operatorname{rad}%
(c)}{\operatorname{cosocle}(m)}\left(  c^{k}-1\right)  <c^{k}-1.
\]
The result now follows since%
\[
c^{k}-\operatorname{rad}\!\left(  c^{k}\left(  c^{k}-1\right)  \right)
>c^{k}-c^{k}+1=1. \qedhere
\]

\end{proof}
An immediate consequence of Theorem~\ref{mainthm1} and Proposition~\ref{abcprop} is the following result.

\begin{corollary}\label{cortri}
If $\left(  1,c-1,c\right)  $ is an $abc$ triple, then $\left(  1,c^{k}%
-1,c^{k}\right)  $ is an $abc$ triple for each positive integer $k$.
\end{corollary}
\noindent In the next section, we will consider further consequences of Theorem~\ref{mainthm1} that do not require knowledge of an $abc$ triple at the start. The proof of Theorem \ref{mainthm1} relies on the
factorization of $c^{k}-1$. A similar factorization holds for $b^{k}+1$ if
$k$ is odd, and our proof of Theorem \ref{mainthm2} makes use of this.

\begin{proof}
[Proof of Theorem \ref{mainthm2}.]Observe that for each positive odd integer $k$, the following equality holds:
$b^{k}+1=\left(  b+1\right)  \sum_{j=0}^{k-1}\left(  -1\right)
^{j}b^{j}$. It follows that $m$ divides $b^{k}+1$ for each positive integer
$k$. By Lemma \ref{Lemmaonrad}, $\operatorname{rad}(b^{k}%
+1)=\operatorname{rad}\!\left(  \frac{b^{k}+1}{\operatorname{cosocle}%
(m)}\right)  $. Since $\frac{\operatorname{rad}(b)}{\operatorname{cosocle}%
(m)}<1$, we observe that%
\[
\operatorname{rad}\!\left(  b^{k}\left(  b^{k}+1\right)  \right)
=\operatorname{rad}(b)\operatorname{rad}\!\left(  \frac{b^{k}+1}%
{\operatorname{cosocle}(m)}\right)  \leq\frac{\operatorname{rad}%
(b)}{\operatorname{cosocle}(m)}\operatorname{rad}(b^{k}+1)<b^{k}+1.
\]
Consequently,%
\[
b^{k}+1-\operatorname{rad}\!\left(  b^{k}\left(  b^{k}+1\right)  \right)
>b^{k}+1-b^{k}-1=0.\qedhere
\]

\end{proof}

Similarly to the deduction of Corollary \ref{cortri}, we now recover the
following result as an immediate consequence of Theorem \ref{mainthm2} and
Proposition \ref{abcprop}.

\begin{corollary}\label{transferkodd}
If $\left(  1,b,b+1\right)  $ is an $abc$ triple, then $\left(  1,b^{k}%
,b^{k}+1\right)  $ is an $abc$ triple for each positive odd integer $k$.
\end{corollary}
Since $(1,8,9)$ is an $abc$ triple, we deduce from Corollary~\ref{transferkodd} that $(1,8^k,8^k + 1)$ is an $abc$ triple for each positive odd integer $k$. We will also recover this sequence of $abc$ triples as a consequence of Corollary~\ref{cor3_11}.

By Corollary \ref{cortri}, we have that if $\left(  1,c-1,c\right)  $ is an
$abc$ triple, then $\left(  1,c^{k}-1,c^{k}\right)  $ is an $abc$ triple for
each positive integer $k$. This leads us to ask: given a positive integer
$c>1$, for what positive integers $k$ is $\left(  1,c^{k}-1,c^{k}\right)  $ an
$abc$ triple? To answer this question, we first recall a few number theory facts. Given a prime number $p$ and an integer $n$,
the $p$\textit{-adic valuation} of~$n$, denoted $v_{p}(n)$, is the unique
integer that satisfies $n=p^{v_{p}(n)}q$ for some integer $q$ that is
relatively prime to $p$. It is easily verified that the following identities hold for each integer $x,y,k$: $v_{p}(xy)=v_{p}(x)+v_{p}(y)$, $v_{p}(\frac{x}{y})=v_{p}(x)-v_{p}(y)$, and $v_{p}(x^k)=kv_{p}(x)$. In order to determine  the exact power of a prime
$p$ that divides $c^{k}-1$, we first consider the following lemma about the binomial coefficient $\binom{y}{j}$.

\begin{lemma}
\label{lemmanew}Let $x,y\geq2$ be integers and let $p$ be a prime that divides
$x$. Then for each $j$ with $2\leq j\leq y$, the following inequality holds:
\[
v_{p}(xy)\leq v_{p}\left(  \binom{y}{j}x^{j}\right)  .
\]
Moreover, equality holds if and only if $\left(  p,j\right)  =\left(
2,2\right)  $ and $x\equiv2\ \operatorname{mod}4$ with $y$ even.
\end{lemma}

\begin{proof}
In the case when $j=2$, we have that
\[
\frac{\binom{y}{2}x^{2}}{xy}=\frac{x(y-1)}{2}.
\]
Since $p$ divides $x$, we have that
\begin{equation}
v_{p}\!\left(  \frac{x(y-1)}{2}\right)  \geq0\qquad\text{if and only if}\qquad
v_{p}(xy)\leq v_{p}\left(  \binom{y}{2}x^{2}\right)  . \label{padiclem}%
\end{equation}
Moreover, equality in Equation (\ref{padiclem}) holds if and only if $p=2$ and
$v_{2}\!\left(  x(y-1)\right)  =1$. Since $x$ is even, it follows that this is
equivalent to $x\equiv2\ \operatorname{mod}4$ and $y$ is even.

Now suppose that $3\leq j\leq y$. It suffices to show that $v_{p}\!\left(
\frac{\binom{y}{j}x^{j}}{xy}\right)  >0$. To this end, observe that%
\[
\frac{\binom{y}{j}x^{j}}{xy}=\binom{y-1}{j-1}\frac{x^{j-1}}{j}.
\]
For each prime $p$, it is always the case that $v_{p}(j)<j-1$. Thus
$v_{p}\!\left(  \frac{x^{j-1}}{j}\right)  >0$ since~$p|x$. The result now
follows since $v_{p}\!\left(  \binom{y-1}{j-1}\right)  \geq0$ and hence
\[
v_{p}\!\left(  \frac{\binom{y}{j}x^{j}}{xy}\right)  \geq v_{p}\!\left(
\frac{x^{j-1}}{j}\right)  >0. \qedhere
\]
\end{proof}

Now suppose that $p$ is a prime that does not divide an integer $n$.
Then the \textit{order} of $n$ modulo~$p$, denoted $\operatorname*{ord}%
_{p}(n)$, is the least positive integer for which $n^{\operatorname*{ord}%
_{p}(n)}\equiv1\ \operatorname{mod}p$. By Fermat's Little Theorem,
$\operatorname*{ord}_{p}(n)$ divides $p-1$. More generally, $n^{k}%
\equiv1\ \operatorname{mod}p$ if and only if $\operatorname*{ord}_{p}(n)$
divides $k$. With this terminology, we now deduce the exact power of a prime
$p$ that divides $c^{k}-1$.

\begin{lemma}
\label{lemmagran1}Let $c$ and $k$ be positive integers with $c>1$. Then $p$
divides $c^{k}-1$ if and only if $\operatorname*{ord}_{p}(c)$ divides $k$.
Moreover, if $p$ divides $c^{k}-1$, then $v_{p}\!\left(  c^{k}-1\right)
=f_{p}+w_{p}$, where $w_{p}=v_{p}\!\left(  k\right)  $ and%
\begin{equation}
f_{p}=\left\{
\begin{array}
[c]{cl}%
v_{2}\!\left(  c^{2}-1\right)  -1 & \text{if }p=2,c\equiv3\ \operatorname{mod}%
4,\text{ and }k\text{ is even,}\\
v_{p}\!\left(  c^{\operatorname*{ord}_{p}(c)}-1\right)  & \text{otherwise.}%
\end{array}
\right.  \label{defoffp}%
\end{equation}

\end{lemma}

\begin{proof}
The statement that $p$ divides $c^{k}-1$ if and only if $\operatorname*{ord}%
_{p}(c)$ divides $k$ is a standard number theory result. So suppose that $p$
divides $c^{k}-1$. Then $\operatorname*{ord}_{p}(c)$ divides $k$ and $p-1$.
The latter is due to Fermat's Little Theorem. In particular,
$\operatorname*{ord}_{p}(c)$ is not divisible by $p$. Note that if $k=1$, then
there is nothing to show. We now proceed by cases and suppose that $k\geq2$.

\textbf{Case 1.} Suppose that $p=2$, $c\equiv3\ \operatorname{mod}4$, and $k$
is even. The result holds if $k=2$, and so we may assume that $k\geq4$. By the
Binomial Theorem,%
\[
c^{k}=\left(  c^{2}\right)  ^{\frac{k}{2}}=\left(  c^{2}-1+1\right)
^{\frac{k}{2}}=1+\frac{k}{2}\left(  c^{2}-1\right)  +\sum_{j=2}^{\frac{k}{2}%
}\binom{\frac{k}{2}}{j}\left(  c^{2}-1\right)  ^{j}.
\]
Since $c^{2}-1\equiv0\ \operatorname{mod}8$, it follows from Lemma
\ref{lemmanew} that%
\[
f_{2}+w_{2}=v_{2}\!\left(  \frac{k}{2}\left(  c^{2}-1\right)  \right)
<v_{2}\!\left(  \binom{\frac{k}{2}}{j}\left(  c^{2}-1\right)  ^{j}\right)  ,
\]
for $2\leq j\leq\frac{k}{2}$. In particular, $c^{k}\equiv1+\frac{k}{2}\left(
c^{2}-1\right)  \ \operatorname{mod}2^{f_{2}+w_{2}+1}\neq 1$ and $c^{k}\equiv
1\ \operatorname{mod}2^{f_{2}+w_{2}}$, from which we conclude that
$v_{2}\!\left(  c^{k}-1\right)  =f_{2}+w_{2}$.

\textbf{Case 2.} Suppose that $p$ is odd or $p=2$ with $c\equiv
1\ \operatorname{mod}4$ or $k$ odd. Write $k=q_{1}p^{w_{p}}\operatorname*{ord}%
_{p}(c)$ for some integer $q_{1}$ that is not divisible by $p$. By the
Binomial Theorem, we obtain%
\begin{align*}
c^{k}=\left(  c^{\operatorname*{ord}_{p}(c)}\right)  ^{q_{1}p^{w_{p}}}  &
=\left(  c^{\operatorname*{ord}_{p}(c)}-1+1\right)  ^{q_{1}p^{w_{p}}}\\
& =1+q_{1}p^{w_{p}}\left(  c^{\operatorname*{ord}_{p}(c)}-1\right)
+\sum_{j=2}^{q_{1}p^{w_{p}}}\binom{q_{1}p^{w_{p}}}{j}\left(
c^{\operatorname*{ord}_{p}(c)}-1\right)  ^{j}.
\end{align*}
From Lemma \ref{lemmanew}, we deduce that
\[
f_{p}+w_{p}=v_{p}\!\left(  q_{1}p^{w_{p}}\left(  c^{\operatorname*{ord}%
_{p}(c)}-1\right)  \right)  <v_{p}\!\left(  \binom{q_{1}p^{w_{p}}}{j}\left(
c^{\operatorname*{ord}_{p}(c)}-1\right)  ^{j}\right)
\]
for $2\leq j\leq q_{1}p^{w_{p}}$. Therefore $c^{k}\equiv1+q_{1}p^{w_{p}%
}\left(  c^{\operatorname*{ord}_{p}(c)}-1\right)  \ \operatorname{mod}%
p^{f_{p}+w_{p}+1} \neq 1$ and $c^{k}\equiv1\ \operatorname{mod}p^{f_{p}+w_{p}}$.
Hence $v_{p}\!\left(  c^{k}-1\right)  =f_{p}+w_{p}$.
\end{proof}

As an immediate consequence of Lemma \ref{lemmagran1} and the Fundamental Theorem of Arithmetic, we obtain the following factorization for $c^{k}-1$.

\begin{corollary}
\label{CorGran}Let $c$ and $k$ be positive integers with $c>1$. Then with
notation as in Lemma \ref{lemmagran1},%
\[
c^{k}-1=\prod_{\operatorname*{ord}_{p}(c)|k}p^{f_{p}+w_{p}}.
\]

\end{corollary}

As a demonstration of Corollary \ref{CorGran}, let $c=21$ and $k=12$. With
notation as above, we see that $w_{2}=2,\ w_{3}=1$, and $w_{p}=0$ for each
prime $p\neq2,3$. Next, we observe that%
\begin{equation}
21^{12}-1=2^{4}\cdot5\cdot11\cdot13\cdot17\cdot61\cdot421\cdot463\cdot
3181.\label{factp}%
\end{equation}
By Lemma \ref{lemmagran1}, the primes appearing in Equation (\ref{factp}) are precisely
those primes $p$ for which $\operatorname*{ord}_{p}(21)$ divides $12$. With a
computer algebra system, such as SageMath \cite{sagemath}, it is checked that
$f_{p}=1$ for each prime $p\neq2$ appearing in Equation (\ref{factp}) and $f_{2}=2$.
Thus, $21^{12}-1=\prod_{\operatorname*{ord}_{p}(21)|12}p^{f_{p}+w_{p}}$. 

\begin{theorem}
\label{gransuggestion}Let $c$ and $k$ be positive integers with $c>1$. With
notation as in Corollary \ref{CorGran}, write%
\[
c^{k}-1=\prod_{\operatorname*{ord}_{p}(c)|k}p^{f_{p}+w_{p}}.
\]
Then $\left(  1,c^{k}-1,c^{k}\right)  $ is an $abc$
triple if and only if one of the following conditions hold:
\begin{itemize}
    \item [$(i)$] there exists a prime $p>\operatorname{rad}(c)$ such that
$\operatorname*{ord}_{p}(c)$ divides $k$ and either $f_{p}\geq2$ or $w_{p}%
\geq1$;

\item [$(ii)$] there exists a prime $p<\operatorname{rad}(c)$ such that
$\operatorname*{ord}_{p}(c)$ divides $k$ and $f_{p}+w_{p}-1\geq m_{p}$, where
$m_{p}$ denote the least positive integer such that $p^{m_{p}}%
>\operatorname{rad}(c)$;

\item [$(iii)$]  for each prime $p$ such that $\operatorname*{ord}%
_{p}(c)$ divides $k$, there exist a non-negative integer $a_{p}\leq f_{p}+w_{p}-1$
such that $\prod_{\operatorname*{ord}_{p}(c)|k}p^{a_{p}}>\operatorname{rad}%
(c)$.
\end{itemize}
\end{theorem}

\begin{proof}
First suppose that $\left(  1,c^{k}-1,c^{k}\right)  $ is an $abc$ triple. By
Proposition \ref{abcprop}, this is equivalent to%
\[
\operatorname{rad}\!\left(  c\right)  <\operatorname{cosocle}\!\left(
c^{k}-1\right)  =\prod_{\operatorname*{ord}_{p}(c)|k}p^{f_{p}+w_{p}-1}.
\]
In particular, taking $a_{p}=f_{p}+w_{p}-1$ yields $\left(  iii\right)  $.

Now suppose there is a prime $p>\operatorname{rad}(c)$ such that
$\operatorname*{ord}_{p}(c)$ divides $k$ and either $f_{p}\geq2$ or $w_{p}%
\geq1$. Note that $f_{p}\geq1$ for each prime $p$ such that
$\operatorname*{ord}_{p}(c)$ divides $k$. Consequently, if $f_{p}\geq2$ or
$w_{p}\geq1$, then $f_{p}+w_{p}\geq2$ and thus $p^{2}$ divides $c^{k}-1$. Then
$\left(  1,c^{k}-1,c^{k}\right)  $ is an $abc$ triple by Theorem
\ref{mainthm1} since $\operatorname{cosocle}\!\left(  p^{2}\right)
=p>\operatorname{rad}\!\left(  c\right)  $.

Next, suppose that there is a prime $p<\operatorname{rad}(c)$ such that
$\operatorname*{ord}_{p}(c)$ divides $k$ and $f_{p}+w_{p}-1\geq m_{p}$. Then
$p^{m_{p}+1}$ divides $c^{k}-1$ and%
\[
\operatorname{cosocle}\!\left(  p^{m_{p}+1}\right)  =p^{m_{p}}%
>\operatorname{rad}\!\left(  c\right)  .
\]
By Theorem \ref{mainthm1}, we deduce that $\left(  1,c^{k}-1,c^{k}\right)  $
is an $abc$ triple.

Lastly, suppose that for each prime $p$ such that $\operatorname*{ord}_{p}(c)$
divides $k$, there exists a positive integer $a_{p}\leq f_{p}+w_{p}-1$ such
that $\prod_{\operatorname*{ord}_{p}(c)|k}p^{a_{p}}>\operatorname{rad}(c)$.
Then $\prod_{\operatorname*{ord}_{p}(c)|k}p^{a_{p}+1}$ divides $c^{k}-1$ and
the result now follows by Theorem \ref{mainthm1} since%
\[
\operatorname{cosocle}\!\left(  \prod_{\operatorname*{ord}_{p}(c)|k}%
p^{a_{p}+1}\right)  =\prod_{\operatorname*{ord}_{p}(c)|k}p^{a_{p}%
}>\operatorname{rad}\!\left(  c\right)  . \qedhere
\]

\end{proof}

As an illustration, consider $c=21$ and $k=12$. In the discussion following
Corollary \ref{CorGran}, we noted that $w_{2}=f_{2}=2$. Moreover, for each
prime $p\neq2$ appearing in Equation (\ref{factp}) we have that $f_{p}=1$ and $w_{p}%
=0$. In particular, we see that statements $\left(  i\right)  $ and $\left(
ii\right)  $ of Theorem \ref{gransuggestion} are not satisfied for each prime
$p$ appearing in Equation (\ref{factp}). We also have that statement $\left(
iii\right)  $ is not satisfied as the only prime for which $f_{p}+w_{p}-1>0$
is $p=2$ and $2^{f_{2}+w_{2}-1}=8<\operatorname{rad}\!\left(  21\right)  $.
It follows that $\left(  1,21^{12}-1,21^{12}\right)  $ is not an $abc$ triple.
In the next section, we will see that $21$ is the first odd integer $n>1$ for
which $\left(  1,n^{\varphi(n)}-1,n^{\varphi(n)}\right)  $ is not an $abc$
triple, where $\varphi(n)$ denotes the Euler-totient function. We note that
$\varphi\!\left(  21\right)  =12$.

\section{Consequences\label{section3}}

In this section, we consider various consequences of Theorems~\ref{mainthm1}
and~\ref{mainthm2}. From these consequences, we deduce the sequences of $abc$
triples that were mentioned in the introduction. We note that this article
began as an investigation of the following question: for what positive odd
integers $n$ is $\left(  1,n^{\varphi(n)}-1,n^{\varphi(n)}\right)  $ an $abc$
triple? Here $\varphi\!\left(  n\right)  $ denotes the Euler-totient function.
The question was motivated by the following observation: if $n$ is an odd
integer such that $3\leq n\leq99$, then $\left(  1,n^{\varphi(n)}%
-1,n^{\varphi(n)}\right)  $ is an $abc$ triple for each $n$ except
$n=21,39,69,$ and $87$. The fact that the four exceptions are composites is no
surprise, as the answer to the question is true for odd primes $n$ \cite{MR4474850}. Our
investigation of this phenomenon led to our Theorems~\ref{mainthm1}
and~\ref{mainthm2}, and our first consequence provides necessary conditions
for when $\left(  1,n^{\varphi(n)}-1,n^{\varphi(n)}\right)  $ is an $abc$
triple for a positive odd integer $n$. To prove this result, we first recall
the following result from elementary number theory.

\begin{lemma}
\label{Lemmaon2k}Let $n$ be a positive odd integer. Then $n^{2^{k}}%
\equiv1\ \operatorname{mod}2^{k+2}$ for each positive integer $k$.
\end{lemma}

\begin{proof}
Since $n$ is odd, there is an integer $m$ such that $n=2m+1$. By the Binomial
Theorem,
{\small
\[ 
n^{2^{k}}=\left(  2m+1\right)  ^{2^{k}}=\sum_{j=0}^{2^{k}}\binom{2^{k}}%
{j}\left(  2m\right)  ^{j}=1+2^{k+1}m\left(  1+\left(  2^{k}-1\right)
m\right)  +\sum_{j=3}^{2^{k}}\binom{2^{k}}{j}\left(  2m\right)  ^{j}.
\]}
Now observe that $m\left(  1+\left(  2^{k}-1\right)  m\right)  $ is always
even and $\binom{2^{k}}{j}\left(  2m\right)  ^{j}$ is divisible by $2^{k+2}$
for $3\leq j\leq2^{k}$. Consequently, $n^{2^{k}}\equiv1\ \operatorname{mod}2^{k+2}$.
\end{proof}

With this result, we obtain our first application of Theorem~\ref{mainthm1}.

\begin{corollary}
\label{coreulerphi}Let $n>1$ be an odd integer and let $\varphi$ denote the
Euler-totient function. Set $d=\gcd(n-1,\varphi(n))$ and $m=2^{v_{2}%
(4\varphi(n))-2v_{2}(d)}d^{2}$. If $\operatorname{cosocle}%
(m)>\operatorname{rad}(n)$, then $\left(  1,n^{\varphi(n)k}-1,n^{\varphi
(n)k}\right)  $ is an $abc$ triple for each positive integer $k$.
\end{corollary}

\begin{proof}
Let $P=\sum_{j=0}^{\varphi(n)-1}n^{j}$ and observe that $n^{\varphi (n)}-1=\left(  n-1\right)  P$. 
Since $d=\gcd\!\left(  n-1,\varphi\!\left(  n\right)  \right)  $ divides $n-1$, we have that $n\equiv
1\ \operatorname{mod}d$ and thus
\[
P\equiv\sum_{j=0}^{\varphi(n)-1}1^{j}\ \operatorname{mod}d=\varphi\!\left(
n\right)  \ \operatorname{mod}d.
\]
In particular, $d$ divides $P$. Since $n^{\varphi(n)}-1=\left(  n-1\right)  P$, we
deduce that $d^{2}$ divides $n^{\varphi(n)}-1$. 

Next, write $\varphi(n) =2^{v_{2}(\varphi(n))}r $ for $r$ an odd integer. By Lemma \ref{Lemmaon2k},%
\[
n^{\varphi(n)}-1=\left(  n^{r}\right)  ^{2^{v_{2}(\varphi(n))}}-1\equiv
0\ \operatorname{mod}2^{v_{2}(\varphi(n))+2}.
\]
Hence $2^{v_{2}(\varphi(n))+2}$ divides $n^{\varphi(n)}-1$. It follows that%
\[
2^{v_{2}(\varphi(n))+2}\frac{d^{2}}{2^{v_{2}(d^{2})}}=2^{v_{2}(4\varphi
(n))-2v_{2}(d)}d^{2}=m
\]
divides $n^{\varphi(n)}-1$. The result now follows from Theorem~\ref{mainthm1}.
\end{proof}

As an illustration, let $n=75$. Then with notation as in Corollary
\ref{coreulerphi}, we observe that $\varphi\!\left(  75\right)=40,\ d=2,$
and $m=32$. Since $\operatorname{cosocle}(32)=16>\operatorname{rad}(75)=15$,
we have that $\left(  1,75^{40k}-1,75^{40k}\right)  $ is an $abc$ triple for
each positive integer $k$. We note that the converse to Corollary
\ref{coreulerphi} does not hold. In fact, if $3\leq n\leq99$ is an odd integer
such that $\left(  1,n^{\varphi(n)}-1,n^{\varphi(n)}\right)  $ is an $abc$
triple, then the corollary fails to show the cases corresponding to
$n=33,35,55,57,63,65,77,93,95,$ and $99$. The following result provides an
improvement, but comes at the cost of having to compute $v_{p}(n^{\varphi
(n)}-1)$ for each prime $p$ that divides $\gcd(n^{\varphi(n)}-1,\varphi(n))$.

\begin{corollary}
\label{coreulerphi1}Let $n>1$ be an integer and let $\varphi$ denote the
Euler-totient function. Set $d=\gcd(n^{\varphi(n)}-1,\varphi(n))$ and%
\[
m=\prod_{p|d}p^{v_{p}(n^{\varphi(n)}-1)}.
\]
If $\operatorname{cosocle}(m)>\operatorname{rad}(n)$, then $\left(
1,n^{\varphi(n)k}-1,n^{\varphi(n)k}\right)  $ is an $abc$ triple for each
positive integer $k$.
\end{corollary}

\begin{proof}
The result follows from
Theorem~\ref{mainthm1} since $m$ divides $n^{\varphi(n)}-1$.
\end{proof}

For odd integers $n$ such that $3\leq n\leq99$, Corollary~\ref{coreulerphi1}
allows us to conclude that for $n\neq21,39,55,57,69,$ and $87$, $\left(
1,n^{\varphi(n)k}-1,n^{\varphi(n)k}\right)  $ is an $abc$ triple for each
positive integer $k$. As noted at the start of the section, $n=21,39,69,$ and
$87$ are the only $n$'s in this range for which $\left(  1,n^{\varphi
(n)}-1,n^{\varphi(n)}\right)  $ is not an $abc$ triple. In particular,
Corollary~\ref{coreulerphi1} fails to show the cases corresponding to $n=55,57$. Indeed, when
$n=55$, we have that $\varphi\!\left(  55\right)  =40$ and $\gcd\!\left(
55^{40}-1,40\right)  =8$. Then $m=2^{v_{2}(55^{40}-1)}=64$, and thus
$\operatorname{cosocle}(64)=32<\operatorname{rad}\!\left(  55\right)  =55$.
Consequently, the assumption of Corollary \ref{coreulerphi1} is not satisfied
in the case when $n=55$. We note that $\operatorname{cosocle}\!\left(
55^{40}-1\right)  =288$, and hence $\left(  1,55^{40k}-1,55^{40k}\right)  $ is
an $abc$ triple for each positive integer $k$ by Proposition \ref{abcprop}.
The failure of Corollaries~\ref{coreulerphi} and \ref{coreulerphi1} in the
$n=55$ case stems from the fact that the primes dividing $m$ must divide
$\varphi(n)$. Indeed, $\operatorname{cosocle}\!\left(  55^{40}-1\right)
=32\cdot9$ and $3\nmid\varphi(55)$.

To state our next result, we recall the Carmichael function $\lambda
:\mathbb{N}\rightarrow\mathbb{N}$, which has the property that $\lambda
\!\left(  m\right)  $ is the least positive integer for which $a^{\lambda
(m)}\equiv1\ \operatorname{mod}m$ for each integer $a$ that is relatively
prime to $m$. In particular, $\lambda\!\left(  m\right)  $ divides
$\varphi\!\left(  m\right)  $.

\begin{corollary}
\label{CarThm}Let $\lambda$ and $\varphi$ denote the Carmichael function and
Euler-totient function, respectively. If $m$ and $n$ are relatively prime
positive integers such that $\operatorname{cosocle}(m)>\operatorname{rad}%
(n)>1$, then $\left(  1,n^{\lambda(m)k}-1,n^{\lambda(m)k}\right)  $ and
$\left(  1,n^{\varphi(m)k}-1,n^{\varphi(m)k}\right)  $ are $abc$ triples for
each positive integer $k$.
\end{corollary}

\begin{proof}
Since $n^{\lambda(m)}\equiv1\ \operatorname{mod}m$, we have that $m$ divides $n^{\lambda(m)}-1$. By Theorem \ref{mainthm1}, we have that $\left(  1,n^{\lambda(m)k}-1,n^{\lambda (m)k}\right)  $ is an $abc$ triple for each positive integer $k$. Since $\lambda (m)\mid \varphi(m)$, we also have that $\left(  1,n^{\varphi(m)k}-1,n^{\varphi(m)k}\right)  $ is an $abc$ triple for each positive integer $k$.
\end{proof}

As an example, choose $n=11$ and $m=32$. Then $\operatorname{cosocle}\!\left(32\right)  =16>\operatorname{rad}\!\left(   11\right)  $, and therefore the conditions of Corollary~\ref{CarThm} are satisfied. As a result, we find that $\left(
1,11^{\lambda(32)k}-1,11^{\lambda(32)k}\right)  =(1,11^{8k}-1,11^{8k})$ is a
sequence of $abc$ triples. More generally, we have the following application
of Corollary \ref{CarThm}.
\begin{corollary}
\label{granville-tucker}Let $n>1$ be an integer and let $p$ be an odd prime such that $p>\operatorname{rad}(n).$ Then for each positive integer $k$, $\left(1,n^{p\left(  p-1\right)  k}-1,n^{p\left(  p-1\right)  k}\right)  $ is an $abc$ triple.
\end{corollary}

\begin{proof}
By assumption, $\operatorname{cosocle}\!\left(  p^{2}\right)
=p>\operatorname{rad}(n)$. Moreover, $\lambda\!\left(  p^{2}\right)  =p\left(
p-1\right)  $ since $p$ is prime. It follows from Corollary \ref{CarThm} that
$\left(  1,n^{\lambda(p^{2})k}-1,n^{\lambda(p^{2})k}\right)  =\left(
1,n^{p\left(  p-1\right)  k}-1,n^{p\left(  p-1\right)  k}\right)  $ is an
$abc$ triple for each positive integer $k$.
\end{proof}

Taking $(n,k)=(2,1)$ in Corollary~\ref{granville-tucker} yields that $\left(  1,2^{p\left(  p-1\right)  }-1,2^{p\left(  p-1\right)  }\right)  $ is an $abc$ triple for each odd prime $p$. This result is originally due to Granville and Tucker~\cite{MR1930670}.
Theorem \ref{gransuggestion} gives the following refinement of Corollary~\ref{granville-tucker}.

\begin{corollary}
\label{corg6}Let $n>1$ be an integer and let $p$ be an odd prime such that
$p>\operatorname{rad}(n)$. Then for each positive integer $k$, $\left(
1,n^{p\operatorname*{ord}_{p}(n)k}-1,n^{p\operatorname*{ord}_{p}(n)k}\right)
$ is an $abc$ triple. In particular, if $n\equiv1\ \operatorname{mod}p$ and
$p>\operatorname{rad}\!\left(  n\right)  $, then $\left(  1,n^{pk}%
-1,n^{pk}\right)  $ is an $abc$ triple for each positive integer $n$.
\end{corollary}

\begin{proof}
In the notation of Theorem \ref{gransuggestion}, we have that $w_{p}%
=v_{p}\!\left(  p\operatorname*{ord}_{p}\!\left(  n\right)  \right)  =1$.
Since $p>\operatorname{rad}\!\left(  n\right)  $ and $\operatorname*{ord}%
_{p}\!\left(  n\right)  $ divides $p\operatorname*{ord}_{p}\!\left(  n\right)
$, Theorem \ref{gransuggestion} $\left(  i\right)  $ implies that $\left(
1,n^{p\operatorname*{ord}_{p}(n)}-1,n^{p\operatorname*{ord}_{p}(n)}\right)  $
is an $abc$ triple. The result now follows by Corollary~\ref{cortri}. The
second statement is automatic since if $n\equiv1\ \operatorname{mod}p$,
then $\operatorname*{ord}_{p}\!\left(  n\right)  =1$.
\end{proof}

As a demonstration, let $n=16$ and $p=5$. Then Corollary \ref{corg6} asserts
that $\left(  1,16^{5k}-1,16^{5k}\right)  $ is an $abc$ triple for each
positive integer $k$.

\begin{corollary}
\label{maincor1}Let $n>1$ be an integer that is either odd or even and
non-squarefree. Then $\left(  1,n^{\left(  n-1\right)  k}-1,n^{\left(
n-1\right)  k}\right)  $ is an $abc$ triple for each positive integer~$k$.
\end{corollary}

\begin{proof}
Let $P=\sum_{j=0}^{n-2}n^{j}$ and observe that $n^{n-1}-1=\left(  n-1\right)
P$. Moreover,
\[
P\equiv\sum_{j=0}^{n-2}\left(  1\right)  ^{j}\ \operatorname{mod}\!\left(
n-1\right)  =0\ \operatorname{mod}\left(  n-1\right)  .
\]
In particular, $\left(  n-1\right)  ^{2}$ divides $n^{n-1}-1$ and thus
\[
\operatorname{rad}(n^{n-1}-1)=\operatorname{rad}\!\left(  \frac{n^{n-1}%
-1}{n-1}\right)  .
\]

Now suppose that $n$ is odd. We claim that $4$ divides $P$. If $n\equiv
1\ \operatorname{mod}4$, then this follows since~$P$ is divisible by $n-1$. So
suppose that $n\equiv3\ \operatorname{mod}4$. Then $4$ divides $n+1$, and hence $4$ divides $P$ since
\[
P\equiv\sum_{j=0}^{n-2}\left(  -1\right)  ^{j}\ \operatorname{mod}\!\left(
n+1\right)  =0\ \operatorname{mod}\left(  n+1\right)  .
\]
Consequently,
\begin{equation}
\operatorname{rad}(n^{n-1}-1)=\operatorname{rad}\!\left(  \frac{n^{n-1}%
-1}{2\left(  n-1\right)  }\right)  \leq\frac{n^{n-1}-1}{2\left(  n-1\right)
}.\label{coreq1}%
\end{equation}
Now observe that by Equation (\ref{coreq1}),%
\[
\operatorname{cosocle}(n^{n-1}-1)=\frac{n^{n-1}-1}{\operatorname{rad}%
(n^{n-1}-1)}\geq2\left(  n-1\right)  >\operatorname{rad}(n).
\]
The claim now follows by Theorem \ref{mainthm1} with $m=n^{n-1}-1$.

Lastly, suppose that $n$ is an even non-squarefree positive integer. Then
$n=a^{2}b$ for some positive integers $a$ and $b$ with $a>1$ and $b$
squarefree. Then $\operatorname{rad}(n)=\operatorname{rad}(ab)\leq ab<n-1$.
Since $\operatorname{rad}(n^{n-1}-1)=\operatorname{rad}\!\left(  \frac
{n^{n-1}-1}{n-1}\right)  \leq\frac{n^{n-1}-1}{n-1}$, we deduce that%
\[
\operatorname{cosocle}(n^{n-1}-1)=\frac{n^{n-1}-1}{\operatorname{rad}%
(n^{n-1}-1)}\geq n-1>\operatorname{rad}\!\left(  n\right)  .
\]
The result follows by Theorem \ref{mainthm1} with $m=n^{n-1}-1$.
\end{proof}
From Corollary~\ref{maincor1}, we recover that $\left(1,9^{k}-1,9^{k}\right)=\left(1,3^{2k}-1,3^{2k}\right)$ is a sequence of $abc$ triples. In particular, we obtain the smallest $abc$ triple $\left(1,8,9\right)$ as a special case.
 Taking $n=8$ in Corollary~\ref{maincor1} gives us the sequence of $abc$ triples $\left(1,8^{7k}-1,8^{7k}\right)$, which generalizes the sequence $\left(  1,8^{7^{k}}-1,8^{7^{k}}\right)  $ that appears in~\cite{MR3584566}.
 
\begin{corollary}
\label{maincor2}Let $n>1$ be an integer. Then $\left(  1,n^{\left(
n+1\right)  k}-1,n^{\left(  n+1\right)  k}\right)  $ is an $abc$ triple
whenever $\left(  n+1\right)  k$ is a positive even integer.
\end{corollary}

\begin{proof}
Let $l$ be a positive even integer and let $P=\sum_{j=0}^{l-1}\left(  -1\right)
^{j+1}n^{j}$. Then $n^{l}-1=\left(  n+1\right)  P$. We now proceed by cases.

\textbf{Case 1.} Suppose that $n$ is a positive even integer and let
$l=2\left(  n+1\right)  $. Since $n \equiv -1\ \operatorname{mod}\!\left(n+1\right)$, we have that $P \equiv \sum_{j=0}^{l-1}\left(  -1\right)
^{j+1} = 0\ \operatorname{mod}\!\left(n+1\right)  $ and thus%
\[
\operatorname{rad}(n^{l}-1)=\operatorname{rad}\!\left(  \frac{n^{l}-1}%
{n+1}\right)  \leq\frac{n^{l}-1}{n+1}.
\]
The claim now holds by Theorem \ref{mainthm1} with $m=n^{l}-1$ since%
\[
\operatorname{cosocle}(n^{l}-1)=\frac{n^{l}-1}{\operatorname{rad}(n^{l}%
-1)}\geq n+1>\operatorname{rad}(n).
\]

\textbf{Case 2.} Suppose that $n$ is a positive odd integer. Then $l=n+1$ is
even and $P\equiv0\ \operatorname{mod}\!\left(  n+1\right)  $. A similar
argument to that of Case $1$ with $m=n^{l}-1$ shows that the result holds by
Theorem~\ref{mainthm1}.
\end{proof}
As an example, choose $n=21$. As a result, $(n+1)k$ is even for every positive integer $k$, and by Corollary~\ref{maincor2} we know that $\left(1,21^{22k}-1,21^{22k}\right)$ is a sequence of $abc$ triples.

\begin{corollary}
\label{coronpowersof2}Let $j\geq2$ be an integer. Then $\left(  1,\left(
2^{j}-1\right)  ^{2k}-1,\left(  2^{j}-1\right)  ^{2k}\right)  $ is an $abc$
triple for each positive integer $k$.
\end{corollary}

\begin{proof}
Observe that $\operatorname{rad}\!\left(  \left(  2^{j}-1\right)  ^{2}\right)
=\operatorname{rad}\!\left(  2^{j}-1\right)  \leq2^{j}-1$. Since $\left(
2^{j}-1\right)  ^{2}-1=2^{j+1}\left(  2^{j-1}-1\right)  $, we deduce that
\[
\operatorname{cosocle}\!\left(  \left(  2^{j}-1\right)  ^{2}-1\right)
=\frac{2^{j+1}\left(  2^{j-1}-1\right)  }{2\operatorname{rad}\!\left(
2^{j-1}-1\right)  }=\frac{2^{j}\left(  2^{j-1}-1\right)  }{\operatorname{rad}%
\!\left(  2^{j-1}-1\right)  }\geq2^{j}.
\]
The result now follows from Theorem~\ref{mainthm1}, since $\operatorname{cosocle}((  2^{j}-1)  ^{2}-1)
>\operatorname{rad}(  (  2^{j}-1)  ^{2})$.
\end{proof}

The $j=2$ and $j=3$ cases in Corollary~\ref{coronpowersof2} result in the sequences of $abc$ triples
$\left(  1,9^{k}-1,9^{k}\right)  $ and $\left(  1,49^{k}-1,49^{k}\right)  $,
respectively. Of note is that the proof of the corollary is made possible by the
lower bound, $\operatorname{cosocle}\!\left(  \left(  2^{j}-1\right)
^{2}-1\right)  \geq2^{j}$. This leads us to ask, can Corollary
\ref{coronpowersof2} be generalized to deduce sequences of $abc$ triples
$\left(  1,c-1,c\right)  $ with $\operatorname{cosocle}\!\left(  c-1\right)  $
bounded below by $n^{j}$ for some positive integer of the form $n^{j}$? The
answer is yes, but we have to take $c=\left(  n^{j}-1\right)  ^{k}$ for some
positive even integer $k$ that is divisible by $n$ to allow a similar argument
to that of Corollary~\ref{coronpowersof2} to work. This is shown below.

\begin{corollary}
\label{Cor3_9} Let $n\geq3$ and $j\geq1$ be integers. If $k$ is a positive
integer such that $nk$ is even, then $\left(  1,\left(  n^{j}-1\right)
^{nk}-1,\left(  n^{j}-1\right)  ^{nk}\right)  $ is an $abc$ triple.
\end{corollary}

\begin{proof}
Observe that $\operatorname{rad}\!\left(  \left(  n^{j}-1\right)
^{nk}\right)  \leq n^{j}-1$ and%
\[
\left(  n^{j}-1\right)  ^{nk}-1=-1+\sum_{l=0}^{nk}\binom{nk}{l}n^{jl}\left(
-1\right)  ^{nk-l}=-kn^{j+1}+\sum_{l=2}^{nk}\binom{nk}{l}n^{jl}\left(
-1\right)  ^{nk-l}.
\]
Note that in the last expression, each term in the sum is divisible by
$n^{j+1}$.
From this, we deduce that $\operatorname{cosocle}\!\left(  \left(
n^{j}-1\right)  ^{nk}-1\right)  \geq n^{j}$. Hence $\operatorname{cosocle}%
\!\left(  \left(  n^{j}-1\right)  ^{nk}-1\right)  >\operatorname{rad}\!\left(
\left(  n^{j}-1\right)  ^{nk}\right)  $, and the result now follows by
Theorem~\ref{mainthm1}.
\end{proof}

As an illustration, consider $(n,j)=(3,1)$ and $k=2l$ for some positive integer $l$. This results in the sequence of
$abc$ triples $\left(  1,64^{l}-1,64^{l}\right)  $.

\begin{corollary}
\label{maincor3}Let $n$ be a positive even integer. Then $\left(
1,n^{\left(  n+1\right)  k},n^{\left(  n+1\right)  k}+1\right)  $ is an $abc$
triple for each positive odd integer $k$.
\end{corollary}

\begin{proof}
Observe that $n^{n+1}+1=\left(  n+1\right)  \sum_{j=0}^{n}\left(  -1\right)
^{j}n^{j}$. Since $n\equiv-1\ \operatorname{mod}\!\left(  n+1\right)  $, it
follows that%
\[
\sum_{j=0}^{n}\left(  -1\right)  ^{j}n^{j}\equiv\sum_{j=0}^{n}%
1\ \operatorname{mod}\!\left(  n+1\right)  =0\ \operatorname{mod}\!\left(
n+1\right).
\]
Hence, $\operatorname{rad}(n^{n+1}+1)=\operatorname{rad}\!\left(
\frac{n^{n+1}+1}{n+1}\right)  \leq\frac{n^{n+1}+1}{n+1}$. Consequently,%
\[
\operatorname{cosocle}(n^{n+1}+1)=\frac{n^{n+1}+1}{\operatorname{rad}%
(n^{n+1}+1)}\geq n+1>\operatorname{rad}\!\left(  n\right)  .
\]
The result now follows from Theorem \ref{mainthm2} by taking
$m=n^{n+1}+1$.
\end{proof}

As a demonstration of the corollary, take $n=22$. Then $\left(1,22^{23k},22^{23k}+1\right)$ is a sequence of $abc$ triples for each positive odd integer~$k$.

\begin{corollary}\label{cor3_11}
Let $n\geq3$ be an odd integer and let $j\geq1$ be an integer. Then 
 for each odd integer~$k$, $\left(
1,\left(  n^{j}-1\right)  ^{nk},\left(  n^{j}-1\right)  ^{nk}+1\right)  $ is
an $abc$ triple.
\end{corollary}

\begin{proof}
Observe that $\operatorname{rad}\!\left(  \left(  n^{j}-1\right)  ^{n}\right)
\leq n^{j}-1$ and%
\[
\left(  n^{j}-1\right)  ^{n}+1=1+\sum_{l=0}^{n}\binom{n}{l}n^{jl}\left(
-1\right)  ^{n-l}=n^{j+1}+\sum_{l=2}^{n}\binom{n}{l}n^{jl}\left(  -1\right)
^{n-l}.
\]
Note that in the last expression, each term in the sum is divisible by
$n^{j+1}$. From this, we conclude that $\operatorname{cosocle}\!\left(
\left(  n^{j}-1\right)  ^{n}+1\right)  \geq n^{j}$. Hence
$\operatorname{cosocle}\!\left(  \left(  n^{j}-1\right)  ^{n}+1\right)
>\operatorname{rad}\!\left(  \left(  n^{j}-1\right)  ^{n}\right)  $, and the
result now follows by Theorem \ref{mainthm2}.
\end{proof}

As an example, let $n=3$ and $j=1$. Then we get the sequence of $abc$ triples
$\left(  1,8^{k},8^{k}+1\right)  $ for each odd integer $k$. In particular, we
recover the $abc$ triple $\left(  1,8,9\right)  $ as a special case.

\section{\texorpdfstring{$abc$}{abc} triples of the form \texorpdfstring{$(1,c-1,c)$}{(1,c-1,c) } and the ABC@Home Project}\label{section4}
The ABC@Home project found that there are exactly $14\hspace{0.15em}482\hspace{0.15em}065$ $abc$
triples $\left(  a,b,c\right)  $ with $c<10^{18}$. The information found by
the ABC@Home project is available on Bart de Smit's webpage \cite{Smit}.
Given an $abc$ triple $\left(  a,b,c\right)  $, we define its \textit{quality}
to~be
\[
q\!\left(  a,b,c\right)  =\frac{\log c}{\log\operatorname{rad}(abc)}.
\]
By definition, we have that
an $abc$ triple $\left(  a,b,c\right)  $ satisfies $\operatorname{rad}\!\left(  abc\right)  <~c$, and thus $q\!\left(  a,b,c\right)
>1$. This gives us the following restatement of the $abc$ conjecture: For each $\epsilon>0$, there are finitely many
$abc$ triples $\left(  a,b,c\right)  $ with $q\!\left(  a,b,c\right)
>1+\epsilon$.

The $abc$ triple with the largest known quality is $\left(  2,3^{10}%
\cdot109,23^{5}\right)  $, which has a quality of approximately $1.6299$. In
fact, Baker's \cite{MR2107944} \textit{explicit }$abc$ \textit{conjecture}
asserts that there is no $abc$ triple $\left(  a,b,c\right)  $ with
$q\!\left(  a,b,c\right)  \geq \frac{7}{4}$. From this statement, Fermat's Last
Theorem for exponent $n>6$ easily follows. We note that the explicit
$abc$ conjecture and the $abc$ conjecture are not equivalent.

\begin{figure}[H]
    \centering
    \includegraphics[scale=.72]{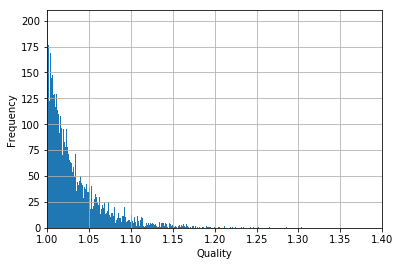}
    \caption{Histogram of the quality of $abc$ triples $(1,c-1,c)$ with~$c<~10^{18}$}
    \label{fig:my_label}
\end{figure}

Let $S\ $denote the set of $abc$ triples of the form $\left(  1,c-1,c\right)
$ with $c<10^{18}$. From the ABC@Home project, we have that $\#S=45\hspace{0.15em}603$.
The largest quality occurring in $S$ corresponds to the $abc$ triple $\left(
1,4374,4375\right)  $, which has quality approximately equal to $1.5679$. 
Figure~\ref{fig:my_label} summarize the distribution of the
quality of all $abc$ triples in $S$. The bin size in the histogram is set to $5000$. We note that all computations done in this section were done on SageMath~\cite{sagemath}, and our code is available on GitHub~\cite{githubprime}.

Table~\ref{tableS} lists the first fifteen $abc$ triples of the form $(1,c-1,c)$, their quality, and whether
they arise from one of the results proven in Section~\ref{section3}. 
The only $abc$ triple in the table that is not of the form $\left(
1,n^{l}-1,n^{l}\right)  $ or $\left(  1,n^{l},n^{l}+1\right)  $ for some
integer $l>1$ is $\left(  1,1215,1216\right)  $. However, most $abc$ triples
in $S$ are not of the aforementioned form. More precisely, $S$ contains a total of
$7376$ (resp. $1038$) $abc$ triples of the form $\left(  1,n^{l}%
-1,n^{l}\right)  $ (resp. $\left(  1,n^{l},n^{l}+1\right)  $) for some integer
$l>1$. We note that $\left(  1,8,9\right)  $ is the only double-counted
element since Mih\u{a}ilescu's Theorem \cite{MR2076124} (formerly known as
Catalan's conjecture) asserts that $2$ and $3$ are the only two consecutive
perfect powers. Consequently,%
\[
T=\left\{  \left(  1,c-1,c\right)  \in S\mid c=n^{l}\text{ or }c=n^{l}+1\text{
for some }l>1\right\}
\]
has $8413$ elements. 
The highest quality $abc$
triple in $T$ is $\left(  1,2400,2401\right)  $, with a quality of
approximately $1.4557$. Observe that this $abc$ triple is obtained from
Corollary~\ref{coronpowersof2} since $\left(  1,2400,2401\right)  =\left(
1,7^{4}-1,7^{4}\right)  $. 

{\renewcommand*{\arraystretch}{1.2}\begin{longtable}{cC{1.05in}c}
\hline
$\left(  1,c-1,c\right)  $ & $q\!\left(  1,c-1,c\right)  $ & Arises from
result in Section \ref{section3}?\\
\hline
\endfirsthead
\caption[]{\emph{continued}}\\
\hline
$\left(  1,c-1,c\right)  $ & $q\!\left(  1,c-1,c\right)  $ & Arises from
result in Section \ref{section3}?\\
\hline
\endhead
\hline
\multicolumn{3}{r}{\emph{continued on next page}}
\endfoot
\hline
\caption{The first fifteen $abc$ triples of the form $(1,c-1,c)$}
\endlastfoot
	
\hline
$\left(  1,8,9\right)  $ & $1.2263$ & \multicolumn{1}{l}{Yes; Corollary
\ref{maincor1} with $\left(  n,k\right)  =\left(  3,1\right)  $}\\\hline
$\left(  1,48,49\right)  $ & $1.0412$ & \multicolumn{1}{l}{Yes; Corollary
\ref{coronpowersof2} with $\left(  j,k\right)  =\left(  3,1\right)  $}\\\hline
$\left(  1,63,64\right)  $ & $1.1127$ & \multicolumn{1}{l}{Yes; Corollary
\ref{Cor3_9} with $\left(  n,j,k\right)  =\left(  3,1,1\right)  $}\\\hline
$\left(  1,80,81\right)  $ & $1.2920$ & \multicolumn{1}{l}{Yes; Corollary
\ref{maincor2} with $\left(  n,k\right)  =\left(  3,1\right)  $}\\\hline
$\left(  1,224,225\right)  $ & $1.0129$ & \multicolumn{1}{l}{Yes; Corollary
\ref{coronpowersof2} with $\left(  j,k\right)  =\left(  4,1\right)  $}\\\hline
$\left(  1,242,243\right)  $ & $1.3111$ & \multicolumn{1}{l}{No}\\\hline
$\left(  1,288,289\right)  $ & $1.2252$ & \multicolumn{1}{l}{No}\\\hline
$\left(  1,512,513\right)  $ & $1.3176$ & \multicolumn{1}{l}{Yes; Corollary
\ref{cor3_11} with $\left(  n,j,k\right)  =\left(  3,1,2\right)  $}\\\hline
$\left(  1,624,625\right)  $ & $1.0790$ & \multicolumn{1}{l}{Yes; Corollary
\ref{maincor1} with $\left(  n,k\right)  =\left(  5,1\right)  $}\\\hline
$\left(  1,675,676\right)  $ & $1.0922$ & \multicolumn{1}{l}{No}\\\hline
$\left(  1,728,729\right)  $ & $1.0459$ & \multicolumn{1}{l}{Yes; Corollary
\ref{maincor1} with $\left(  n,k\right)  =\left(  3,3\right)  $}\\\hline
$\left(  1,960,961\right)  $ & $1.0048$ & \multicolumn{1}{l}{Yes; Corollary
\ref{coronpowersof2} with $\left(  j,k\right)  =\left(  5,1\right)  $}\\\hline
$\left(  1,1024,1025\right)  $ & $1.1523$ & \multicolumn{1}{l}{Yes; Corollary
\ref{maincor3} with $\left(  n,k\right)  =\left(  4,1\right)  $}\\\hline
$\left(  1,1215,1216\right)  $ & $1.1194$ & \multicolumn{1}{l}{No}\\\hline
$\left(  1,2303,2304\right)  $ & $1.0204$ & \multicolumn{1}{l}{No}
\label{tableS}
\end{longtable}}

Now suppose that $\left(  1,n^{l}-1,n^{l}\right)  $ is an $abc$ triple for
some integer $l>1$. By Proposition \ref{abcprop}, we know that
$\operatorname{cosocle}\!\left(  n^{l}-1\right)  >\operatorname{rad}\!\left(
n\right)  $. However, checking that $\left(  1,n^{l}-1,n^{l}\right)  $ is an
$abc$ triple via this criteria gets more difficult as $n^{l}$ grows. By
Theorem \ref{mainthm1}, we can deduce that $\left(  1,n^{l}-1,n^{l}\right)  $
is an $abc$ triple if there is a divisor $m$ of $n^{l}-1$ such that
$\operatorname{cosocle}\!\left(  m\right)  >\operatorname{rad}\!\left(
n\right)  $. By considering those elements in $T$ of the form $\left(
1,n^{l}-1,n^{l}\right)  $ for some integer $l>1$, we find that $m$ can be
taken to be a proper divisor of $n^{l}-1$, except for the $abc$ triples
$\left(  1,c-1,c\right)  $ where $c\in\left\{
9,676,11309769,17380816062160329\right\}  $. Indeed, $\operatorname{rad}%
\!\left(  676\right)  =26$ and $675=3^{3}5^{2}$. The only divisor of $675$
satisfying $\operatorname{cosocle}\!\left(  m\right)  >26$ is $m=675$. 

The above leads us to ask: given $\left(  1,n^{l}-1,n^{l}\right)  \in T$ with $l>1$ an
integer, what is the least divisor~$m$ of $n^{l}-1$ for which
$\operatorname{cosocle}\!\left(  m\right)  >\operatorname{rad}\!\left(
n\right)  $? Using SageMath \cite{sagemath}, we answered this question, and
our datafile can be accessed in \cite[triples\_for\_thm1.csv]{githubprime}. Table~\ref{tableT1} gives the
first fifteen elements $\left(  a,b,c\right)  $ in $T$ of the form $\left(
1,n^{l}-1,n^{l}\right)  $, where $n$ and $l$ are listed, as well as the least
divisor $m$ of $n^{l}-1$ for which $\operatorname{cosocle}\!\left(  m\right)
>\operatorname{rad}\!\left(  n\right)  $ holds. The quality of the $abc$
triple is also given.

{\renewcommand*{\arraystretch}{1.2}\begin{longtable}{C{1in}C{.25in}C{.25in}C{.25in}C{.75in}}
\hline
$\left(  a,b,c\right)  $ & $n$ & $l$ & $m$ & $q\!\left(  a,b,c\right)$\\
\hline
\endfirsthead
\caption[]{\emph{continued}}\\
\hline
$\left(  a,b,c\right)  $ & $n$ & $l$ & $m$ & $q\!\left(  a,b,c\right)$\\
\hline
\endhead
\hline
\multicolumn{3}{r}{\emph{continued on next page}}
\endfoot
\hline
\caption{The first fifteen $abc$ triples $\left(a,b,c\right)  $ of the form $\left(1,n^{l}-1,n^{l}\right)  $ for $l>1$, with $m$ the least divisor of $n^l-1$ satisfying $\operatorname{cosocle}\!\left(m  \right)>\operatorname{rad}\!\left(  n\right)  $}
\endlastfoot
	
\hline
$\left(  1,8,9\right)  $ & $3$ & $2$ & $8$ & $1.2263$\\\hline
$\left(  1,48,49\right)  $ & $7$ & $2$ & $16$ & $1.0412$\\\hline
$\left(  1,63,64\right)  $ & $2$ & $6$ & $9$ & $1.1127$\\\hline
$\left(  1,80,81\right)  $ & $3$ & $4$ & $8$ & $1.2920$\\\hline
$\left(  1,224,225\right)  $ & $15$ & $2$ & $32$ & $1.0129$\\\hline
$\left(  1,242,243\right)  $ & $3$ & $5$ & $121$ & $1.3111$\\\hline
$\left(  1,288,289\right)  $ & $17$ & $2$ & $144$ & $1.2252$\\\hline
$\left(  1,624,625\right)  $ & $5$ & $4$ & $16$ & $1.0790$\\\hline
$\left(  1,675,676\right)  $ & $26$ & $2$ & $675$ & $1.0922$\\\hline
$\left(  1,728,729\right)  $ & $3$ & $6$ & $8$ & $1.0459$\\\hline
$\left(  1,960,961\right)  $ & $31$ & $2$ & $64$ & $1.0048$\\\hline
$\left(  1,2303,2304\right)  $ & $48$ & $2$ & $49$ & $1.0204$\\\hline
$\left(  1,2400,2401\right)  $ & $7$ & $4$ & $16$ & $1.4557$\\\hline
$\left(  1,3024,3025\right)  $ & $55$ & $2$ & $432$ & $1.0348$\\\hline
$\left(  1,3968,3969\right)  $ & $63$ & $2$ & $64$ & $1.1554$
\label{tableT1}
\end{longtable}}

Similarly, we ask the same question in the setting of Theorem \ref{mainthm2}.
That is, given $\left(  1,n^{l},n^{l}+1\right)  \in T$ with $l>1$ an odd
integer, what is the least positive divisor $m$ of $n^{l}+1$ for which
$\operatorname{cosocle}\!\left(  m\right)  >\operatorname{rad}\!\left(
n\right)  $? We note that $T$ has $596$ elements of the form $\left(
1,n^{l},n^{l}+1\right)  $ for some integer $l>1$. We also answer this question
through SageMath, and our datafile is found in \cite[triples\_for\_thm2.csv]{githubprime}. Table~\ref{tableT2} gives the first fifteen elements $\left(  a,b,c\right)  $ in $T$ of the
form $\left(  1,n^{l},n^{l}+1\right)  $, where $n$ and $l$ are listed, as well
as the least divisor $m$ of $n^{l}+1$ for which $\operatorname{cosocle}\!\left(
m\right)  >\operatorname{rad}\!\left(  n\right)  $ holds. 
In particular, we find that $(1,8,9)$ is the only $abc$ triple of the form $(1,n^l,n^l+1)$ in $T$ with $l>1$ an odd integer for which there is no proper divisor $m$ of $n^l+1$ satisfying $\operatorname{cosocle}\!\left(  m\right)  >\operatorname{rad}\!\left(
n\right)  $.

{\renewcommand*{\arraystretch}{1.2}\begin{longtable}{C{1.75in}C{.25in}C{.25in}C{.25in}C{.75in}}
\hline
$\left(  a,b,c\right)  $ & $n$ & $l$ & $m$ & $q\!\left(  a,b,c\right)$\\
\hline
\endfirsthead
\caption[]{\emph{continued}}\\
\hline
$\left(  a,b,c\right)  $ & $n$ & $l$ & $m$ & $q\!\left(  a,b,c\right)$\\
\hline
\endhead
\hline
\multicolumn{3}{r}{\emph{continued on next page}}
\endfoot
\hline
\caption{The first fifteen $abc$ triples $\left(a,b,c\right)  $ of the form $\left(1,n^{l},n^{l}+1\right)  $ for $l>1$ an odd integer, with $m$ the least divisor of $n^l+1$ satisfying $\operatorname{cosocle}\!\left(m  \right)>\operatorname{rad}\!\left(  n\right)  $}
\endlastfoot
	
\hline
$\left(  1,8,9\right)  $ & $2$ & $3$ & $9$ & $1.2263$\\\hline
$\left(  1,512,513\right)  $ & $2$ & $9$ & $9$ & $1.3176$\\\hline
$\left(  1,6859,6860\right)  $ & $19$ & $3$ & $343$ & $1.2281$\\\hline
$\left(  1,12167,12168\right)  $ & $23$ & $3$ & $676$ & $1.2555$\\\hline
$\left(  1,17576,17577\right)  $ & $26$ & $3$ & $81$ & $1.0039$\\\hline
$\left(  1,29791,29792\right)  $ & $31$ & $3$ & $784$ & $1.1424$\\\hline
$\left(  1,32768,32769\right)  $ & $2$ & $15$ & $9$ & $1.0406$\\\hline
$\left(  1,110592,110593\right)  $ & $48$ & $3$ & $49$ & $1.0135$\\\hline
$\left(  1,250047,250048\right)  $ & $63$ & $3$ & $64$ & $1.0351$\\\hline
$\left(  1,279936,279937\right)  $ & $6$ & $7$ & $49$ & $1.0124$\\\hline
$\left(  1,512000,512001\right)  $ & $80$ & $3$ & $81$ & $1.4433$\\\hline
$\left(  1,1953125,1953126\right)  $ & $5$ & $9$ & $27$ & $1.0423$\\\hline
$\left(  1,2097152,2097153\right)  $ & $2$ & $21$ & $9$ & $1.0287$\\\hline
$\left(  1,3176523,3176524\right)  $ & $147$ & $3$ & $676$ & $1.0145$\\\hline
$\left(  1,7077888,7077889\right)  $ & $192$ & $3$ & $169$ & $1.0515$
\label{tableT2}
\end{longtable}}

Next, we investigate how many elements of $T$ arise from the results proven in
Section~\ref{section3}. Indeed, each $abc$ triple produced by the results of
that section are of the form $\left(  1,n^{l}-1,n^{l}\right)  $ or $\left(
1,n^{l},n^{l}+1\right)  $ for some integer $l>1$. Moreover, for each $abc$
triple obtained from one of our corollaries in Section~\ref{section3}, we apply
the following result from \cite[Section~2.3]{Horst}.

\begin{proposition}
\label{transfer_method}Let $\left(  1,c-1,c\right)  $ be an $abc$ triple. Then
the following are $abc$ triples:
\[
\left(  1,\left(  c-1\right)  ^{3},c\left(  c^{2}-3c+3\right)  \right)
\qquad\text{and}\qquad\left(  1,c\left(  c-2\right)  ,\left(  c-1\right)
^{2}\right)  .
\]

\end{proposition}

As a demonstration, the $abc$ triple $\left(  1,2303,2304\right)  $ is obtained
from the $abc$ triple $\left(  1,48,49\right)  $ since $2304=48^{2}$. In
particular, $\left(  1,2303,2304\right)  $ can now be viewed as a consequence
of Corollary~\ref{coronpowersof2} and Proposition~\ref{transfer_method}.
Proposition~\ref{transfer_method} is part of a more general result in
\cite[Section~2.3]{Horst}, which provides a way of mapping an $abc$ triple $\left(
a,b,c\right)  $ to a new $abc$ triple by applying polynomial identities. The
more general result arises by splitting the binomial formula $\left(
a+b\right)  ^{n}$ to obtain the following family of identities:%

\[
a^{n-k}\left(  \sum_{j=0}^{k}\binom{n}{j}a^{k-j}b^{j}\right)  +b^{k+1}\left(
\sum_{j=0}^{n-k-1}\binom{n}{j}a^{j}b^{n-k-1-j}\right)  =c^{n}.
\]
Taking $k=0$ yields Corollary~\ref{cortri}. Therefore, the two non-trivial
polynomial identities with $a=1$ are those occurring in
Proposition~\ref{transfer_method}.

Corollaries~\ref{coreulerphi1} through~\ref{cor3_11} provide us with a recipe for constructing $abc$ triples.
For each of these corollaries, we consider the set
\[
C_{i}=\left\{  \left(  1,c-1,c\right)  \in T\mid\left(  1,c-1,c\right)  \text{
is obtained from Corollary 3.i}\right\}  ,
\]
where $3\leq i\leq 12$. By Table \ref{tableS}, we see that $\left(
1,224,225\right)  \in C_{9}$, but $\left(  1,242,243\right)  \not \in C_{i}$
for each $i$. Using SageMath, we have the following table:
\[%
{\renewcommand*{\arraystretch}{1.2}\begin{tabular}
[c]{ccccccccccc}%
$i$ & $3$ & $4$ & $5$ & $6$ & $7$ & $8$ & $9$ & $10$ & $11$ & $12$\\\hline
$\#C_{i}$ & $32$ & $58$ & $12$ & $17$ & $41$ & $29$ & $81$ & $46$ & $18$ & $36$%
\end{tabular}.}
\]
The low number of $abc$ triples in $T$ occurring in each $C_{i}$ is expected.
Indeed, for Corollary~\ref{granville-tucker} to yield an $abc$ triple in $T$, we require
that $n>1$ be an integer, $p$ be an odd prime such that $p>\operatorname{rad}%
\!\left(  n\right)  $, and $n^{p\left(  p-1\right)  k}<10^{18}$ for some
integer $k$. For~$n$ an odd integer, the only possible $\left(  n,p,k\right)
$ is $\left(  3,5,1\right)  $, which gives the $abc$ triple $\left(
1,3486784400,3486784401\right)  $. We also note that since Corollary~\ref{granville-tucker} is a special case of Corollary~\ref{CarThm}, we have
that $C_{5}\subseteq C_{4}$. Now let%
\[
C=\bigcup_{3\leq i\leq12}C_{i}.
\]
We find that $\#C=164$. 

Lastly, let $D$ be the set of $abc$ triples in $T$
with the property that an element of $D$ is in $C$ or can be obtained from an $abc$ triple in $C$
after successive applications of Proposition~\ref{transfer_method} and Corollaries~\ref{cortri} and \ref{transferkodd}. As an
illustration, the $abc$ triple $\left(  1,12214672127,12214672128\right)$ is not in
$C$, but it is in $D$. To see this, recall that $\left(
1,2303,2304\right)  $ is obtained from the $abc$ triple $\left(
1,48,49\right)  $ via Proposition~\ref{transfer_method}. Then,%
\[
\left(  1,12214672127,12214672128\right)  =\left(  1,\left(  c-1\right)
^{3},c\left(  c^{2}-3c+3\right)  \right)  ,
\]
where $c=2304$, which shows that the $abc$ triple is in $D$. In fact, with the exception of the $abc$ triple $(1, 1215, 1216)$, every $abc$ triple appearing in Table~\ref{tableS} is in $D$. Using SageMath,
we find that $D$ has $311$ elements. 

We conclude this article by considering the percentage of $abc$ triples
$\left(  1,c-1,c\right)  $ in $S$ and $T$, that are also in $D$. More
precisely, for sets $X$ and $Y$ such that $X\subseteq Y\subseteq S$, we define%
\[
\delta_{X,Y}\!\left(  x\right)  =\frac{\#\left\{  \left(  1,c-1,c\right)  \in
X\mid c\leq x\right\}  }{\#\left\{  \left(  1,c-1,c\right)  \in Y\mid c\leq
x\right\}  }.
\]
In particular, $\delta_{X,Y}\!\left(  x\right)  $ gives the percentage of
$abc$ triples $\left(  1,c-1,c\right)  $ of $Y$ with $c\leq x$ that are in $X$. The table
below gives some values of $\delta_{T,S}\!\left(  x\right)  ,\delta
_{D,S}\!\left(  x\right)  $, and $\delta_{D,T}\!\left(  x\right)  $.
\[%
{\renewcommand*{\arraystretch}{1.2}\begin{tabular}
[c]{ccccccccc}%
$x$ & $10^{4}$ & $10^{6}$ & $10^{8}$ & $10^{10}$ & $10^{12}$ & $10^{14}$ &
$10^{16}$ & $10^{18}$\\\hline
$\delta_{T,S}\!\left(  x\right)  $ & $80\%$ & $57.8\%$ & $45.2\%$ & $35.1\%$ &
$30.0\%$ & $24.6\%$ & $20.9\%$ & $18.4\%$\\\hline
$\delta_{D,S}\!\left(  x\right)  $ & $53.3\%$ & $28.1\%$ & $13.5\%$ & $7.03\%$
& $3.79\%$ & $2.06\%$ & $1.14\%$ & $0.68\%$\\\hline
$\delta_{D,T}\!\left(  x\right)  $ & $66.7\%$ & $48.7\%$ & $29.9\%$ & $20.0\%$
& $12.6\%$ & $8.40\%$ & $5.47\%$ & $3.70\%$ \\\hline
\end{tabular}}
\]
In particular, we see that $D$ contains nearly half of the $abc$ triples
$\left(  1,c-1,c\right)  $ in $T$ with $c\leq10^{6}$. This aligns with our
earlier observation that with the exception of $\left(  1,1215,1216\right)  $,
each $abc$ triple in Table \ref{tableS} is in $D$. However, for $c\geq10^{8}$,
the percentage of $abc$ triples in $S$ that are in $T$ begins to decrease
rapidly which leads us to conclude that most $abc$ triples in $T$ and $S$ do
not fall into families such as those illustrated in Section \ref{section3}.

\vspace{1em}

\noindent {\bf Acknowledgements.} The authors would like to thank the National Science Foundation, Pomona College, Edray Goins, Renee Bell, Cory Colbert, Bianca Thompson, and the staff and students of the Pomona Research in Mathematics Experience (PRiME) for their support and camaraderie as this work was being undertaken. Research at PRiME was supported by the National Science Foundation award DMS-2113782. The authors also thank Andrew Granville for his comments and suggestions on an earlier preprint. In particular, we are grateful for his observations regarding those integers $k$ for which $(1,c^k - 1, c^k)$ is an $abc$ triple. This led to the deduction of Theorem~\ref{gransuggestion}. The authors are grateful to the referee for their comments and suggestions, and for catching a small error in the $p=2$ case of Lemma~\ref{lemmagran1} of the original submission. We also would like to thank the High Performance Computing Program team at California State University San Bernardino and the National Research Platform, especially Youngsu Kim, for providing us access to SageMath in said computing program. The High Performance Computing Program is supported in part by National Science Foundation awards CNS-1730158, ACI-1540112, ACI-1541349, OAC-1826967, OAC-2112167, CNS-2120019, the University of California Office of the President, and the University of California San Diego’s California Institute for Telecommunications and Information Technology/Qualcomm Institute. Thanks to CENIC for the 100Gbps networks.

\bibliographystyle{amsalpha}
\bibliography{bibliography}

\newcommand{\etalchar}[1]{$^{#1}$}
\providecommand{\bysame}{\leavevmode\hbox to3em{\hrulefill}\thinspace}
\providecommand{\MR}{\relax\ifhmode\unskip\space\fi MR }
\providecommand{\MRhref}[2]{%
  \href{http://www.ams.org/mathscinet-getitem?mr=#1}{#2}
}
\providecommand{\href}[2]{#2}
\begin{thebibliography}{ASBHS23}

\bibitem[ASBHS23]{githubprime}
Elise Alvarez-Salazar, Alexander~J. Barrios, Calvin Henaku, and Summer Soller,
  \emph{Code for $abc$ triples of the form $(1,c-1,c)$},
  \url{https://github.com/alexanderbarrios/abc_triples/}, 2023.

\bibitem[Bak04]{MR2107944}
Alan Baker, \emph{Experiments on the {$abc$}-conjecture}, Publ. Math. Debrecen
  \textbf{65} (2004), no.~3-4, 253--260. \MR{2107944}

\bibitem[Bar23]{MR4474850}
Alexander~J. Barrios, \emph{Good elliptic curves with a specified torsion
  subgroup}, J. Number Theory \textbf{242} (2023), 21--43. \MR{4474850}

\bibitem[dS]{Smit}
Bart de~Smit, \emph{Abc-triples}, { \tt
  \url{http://www.math.leidenuniv.nl/~desmit/abc/} }.

\bibitem[Elk91]{MR1141316}
Noam~D. Elkies, \emph{{$ABC$} implies {M}ordell}, Internat. Math. Res. Notices
  (1991), no.~7, 99--109. \MR{1141316}

\bibitem[GT02]{MR1930670}
Andrew Granville and Thomas~J. Tucker, \emph{It's as easy as {$abc$}}, Notices
  Amer. Math. Soc. \textbf{49} (2002), no.~10, 1224--1231. \MR{1930670}

\bibitem[Lan90]{MR1005184}
Serge Lang, \emph{Old and new conjectured {D}iophantine inequalities}, Bull.
  Amer. Math. Soc. (N.S.) \textbf{23} (1990), no.~1, 37--75. \MR{1005184}

\bibitem[Mas17]{MR3731300}
D.~W. Masser, \emph{Abcological anecdotes}, Mathematika \textbf{63} (2017),
  no.~3, 713--714. \MR{3731300}

\bibitem[Mih04]{MR2076124}
Preda Mih\u{a}ilescu, \emph{Primary cyclotomic units and a proof of {C}atalan's
  conjecture}, J. Reine Angew. Math. \textbf{572} (2004), 167--195.
  \MR{2076124}

\bibitem[MM16]{MR3584566}
Greg Martin and Winnie Miao, \emph{{$abc$} triples}, Funct. Approx. Comment.
  Math. \textbf{55} (2016), no.~2, 145--176. \MR{3584566}

\bibitem[Oes88]{MR992208}
Joseph Oesterl\'e, \emph{Nouvelles approches du ``th\'eor\`eme'' de {F}ermat},
  Ast\'erisque (1988), no.~161-162, Exp.\ No.\ 694, 4, 165--186 (1989),
  S\'eminaire Bourbaki, Vol. 1987/88. \MR{992208}

\bibitem[Pal14]{palen}
Willem~Jan Palenstijn, \emph{Finding abc-triples using elliptic curves}, 2014,
  Thesis (Ph.D.)--Universiteit Leiden.

\bibitem[S{\etalchar{+}}23]{sagemath}
W.\thinspace{}A. Stein et~al., \emph{{S}age {M}athematics {S}oftware ({V}ersion
  9.7)}, The Sage Development Team, 2023, {\tt http://www.sagemath.org}.

\bibitem[Ste84]{MR781193}
C.~L. Stewart, \emph{A note on the product of consecutive integers}, Topics in
  classical number theory, {V}ol. {I}, {II} ({B}udapest, 1981), Colloq. Math.
  Soc. J\'{a}nos Bolyai, vol.~34, North-Holland, Amsterdam, 1984,
  pp.~1523--1537. \MR{781193}

\bibitem[vdH10]{Horst}
Johannes~Petrus van~der Horst, \emph{Finding abc-triples using elliptic
  curves}, 2010, Masters Thesis --Universiteit Leiden.

\end{thebibliography}
\end{document}